\documentclass[11pt,a4paper]{article} 
\usepackage{amssymb,amsfonts,amsmath} 
\usepackage{graphicx}                 
\usepackage{enumerate}                %
\usepackage{color}
\newtheorem{ter}{Theorem}[section]
\newtheorem{hyp}{Conjecture}[section]

\newtheorem{pr}{Proposition}[section]
\newtheorem{cor}{Corollary}[section]
\newtheorem{lm}{Lemma}[section]
\newtheorem{rem}{Remark}[section]
\newtheorem{ex}{Example}[section]

\newenvironment{proof}{\medskip\noindent{\it Proof. }}
{

\vspace{-0.4cm} $\Box$ {\raggedleft   

}} 

\begin{document}

\title{Donkin-Koppinen filtration for general linear supergroups}
\author{R. La Scala and A.N.Zubkov}
\date{}

\maketitle
\begin{abstract}
We consider a generalization of Donkin-Koppinen filtrations for
coordinate superalgebras of general linear supergroups. More
precisely, if $G=GL(m|n)$ is a general linear supergroup of
(super)degree $(m|n)$, then its coordinate superalgebra $K[G]$ is
a natural $G\times G$-supermodule. For every finitely generated
ideal $\Gamma\subseteq \Lambda\times\Lambda$,
the largest
subsupermodule $O_{\Gamma}(K[G])$ of $K[G]$, which has all composition factors
of the form $L(\lambda)\otimes L(\mu)$ where $(\lambda, \mu)\in\Gamma$, has
a decreasing filtration $O_{\Gamma}(K[G])=V_0\supseteq
V_1\supseteq\ldots$ such that $\bigcap_{t\geq 0}V_t=0$ and
$V_t/V_{t+1}\simeq V_-(\lambda_t)^*\otimes H_-^0(\lambda_t)$ for each $t\geq 0$. Here
$H_-^0(\lambda)$ is a costandard $G$-supermodule, and $V_-(\lambda)$ is a standard
$G$-supermodule, both of highest weight $\lambda\in\Lambda$ (see \cite{z}). We deduce the existence of such a filtration from
more general facts about standard and costandard filtrations in certain highest weight categories which will be
proved in Section 4.
Until now, analogous results were known only for highest weight categories with finite sets of weights.
We believe that the reader will find the results of Section 4 interesting on its own.  
Finally, we apply our main result to describe invariants of (co)adjoint action of $G$.

\end{abstract}

\section*{Introduction}

Let $G$ be a reductive algebraic group defined over an
algebraically closed field $K$. The group $G\times G$ acts on $G$
by $(g, (g_1, g_2))\mapsto g_1^{-1}gg_2$ for $g, g_1, g_2\in G$. This
action induces a structure of a rational $G$-bimodule on $K[G]$. Donkin 
proved first (and Koppinen observed this in \cite{kop} as a remark) that $K[G]$ has an increasing
$G$-bimodule filtration $0\subseteq V_1\subseteq
V_2\subseteq\ldots$ such that $\bigcup_{t\geq 1}V_t=K[G]$ and
$V_t/V_{t-1}\simeq V(\lambda_t)^*\otimes H^0(\lambda_t)$ for
$\lambda_t\in X(T)^+$. Additionally, if $k > l$, then either
$\lambda_k \leq\lambda_l$ or $\lambda_k$ is not comparable to $\lambda_l$.

In the present article we generalize this result to the case of the general
linear supergroup $G=GL(m|n)$. In contrast to the classical case,
the $G$-superbimodule $K[G]$ has no increasing filtration as
above. In fact, a set of highest weights $\Lambda$ of simple
$G$-supermodules has no minimal elements. Nevertheless, one can
prove that for any finitely generated ideal $\Gamma\subseteq
\Lambda\times\Lambda$, the largest subsupermodule
$O_{\Gamma}(K[G])$ of $K[G]$, which has all composition factors of the form
$L(\lambda)\otimes L(\mu)$ where $(\lambda, \mu)\in\Gamma$, has a
decreasing filtration $O_{\Gamma}(K[G])=V_0\supseteq
V_1\supseteq\ldots$ such that $\bigcap_{t\geq 0}V_t=0$ and
$V_t/V_{t+1}\simeq V_-(\lambda_t)^*\otimes H_-^0(\lambda_t)$ for each $t\geq 0$. Here
$H_-^0(\lambda)=ind^G_{B^-} K_{\lambda}$ is an induced supermodule
of highest weight $\lambda$, $V_-(\lambda)$ is a Weyl supermodule
of highest weight $\lambda$, $B^-$ is a Borel subsupergroup of $G$
consisting of all lower triangular matrices, and $K_{\lambda}$ is a
natural (even) one-dimensional $B^-$-supermodule of weight
$\lambda$ (see \cite{z} for more details). Moreover, $V_-(\lambda)^*\simeq 
H^0_+(-\lambda)=ind^G_{B^+} K_{-\lambda}$, where $B^+$ is the opposite Borel subsupergroup
but, in general, $V_-(\lambda)^*$ is not isomorphic to any $H^0_-(\mu)$. Therefore we have to work with
induced and Weyl supermodules corresponding to both Borel subsupergroups as well as with both
orderings of $\Lambda$.  
In the last section
we use the above filtrations to describe invariants of the (co)adjoint action
of $G$.

\section{Preliminary definitions and notations}

Let $G$ be an affine supergroup. In other words, $G$ is a
(representable) functor from the category of commutative
superalgebras $SAlg_K$ to the category of groups $Gr$ such that
$G(A)=Hom_{SAlg_K}(K[G], A)$ for $A\in SAlg_K$. The Hopf superalgebra
$K[G]$ is called the {\it coordinate} superalgebra of $G$. The category
of (left) $G$-supermodules with even morphisms will be denoted by
$G-mod$. The category $G-mod$ is equivalent to the category of
(right) $K[G]$-supercomodules with even morphisms (see \cite{bkl, z}). If
$V\in G-mod$, then its coaction map will be denoted by $\rho_V$ and we will 
write it using Sweedler's notation as $\rho_V(v)=\sum v_1\otimes f_2$ for  
$v,v_1\in V$ and $f_2\in K[G]$.

In what follows, let $Hom_G(V, W)$ be a (super)space of all (not
necessarily even) morphisms between $G$-supermodules $V$ and $W$. In
particular, $Hom_{G-mod}(V, W)= Hom_G(V, W)_0$. In the same way, let
$Ext^i_G(V, ?)=R^i Hom_G(V, ?)$ for $i\geq 0$. Every $Ext^i_G(V, W)$ has
a superspace structure given by $Ext^i_G(V, W)_{\epsilon}=R^i Hom_G(V,
W)_{\epsilon}$ for $\epsilon=0, 1$. Moreover, $Ext^i_G(V^c, W)\simeq
Ext^i_G(V, W^c)=Ext^i_G(V, W)^c$, where $V^c$ is conjugated to $V$ (see \cite{z}).

Let $K[c_{ij}|1\leq i, j\leq m+n]$ be a commutative superalgebra
freely generated by elements $c_{ij}$, where $|c_{ij}|=0$ if 
$1\leq i, j\leq m$ or $m+1\leq i, j\leq m+n$, and $|c_{ij}|=1$ otherwise. 
Denote a generic matrix $(c_{ij})_{1\leq i, j\leq
m+n}$ by $C$ and put
$$C_{00}=(c_{ij})_{1\leq i, j\leq m}, C_{01}=(c_{ij})_{1\leq i\leq m, m+1\leq j\leq
m+n},$$ $$C_{10}=(c_{ij})_{m+1\leq i\leq m+n, 1\leq j\leq n},
C_{11}=(c_{ij})_{m+1\leq i\leq m+n, m+1\leq j\leq m+n},$$ and
$d_1=\det(C_{00}), d_2=\det(C_{11})$. The general linear
supergroup $GL(m|n)$ is an algebraic supergroup whose coordinate
superalgebra $K[GL(m|n)]$ is isomorphic to $K[c_{ij}|1\leq i,
j\leq m+n]_{d_1, d_2}$. The comultiplication and the counit of
$K[GL(m|n)]$ are given by $\delta_{GL(m|n)}(c_{ij})=\sum_{1\leq
t\leq m+n}c_{it}\otimes c_{tj}$ and 
$\epsilon_{GL(m|n)}(c_{ij})=\delta_{ij}$, respectively. For a
definition of the antipode see \cite{z, ber}.

Let $\cal C$ be a $K$-abelian and locally artinian Grothendieck
category (see \cite{cps, bd} for definitions). Assume that all
simple objects in $\cal C$ are indexed by elements of a
partially ordered set $(\Lambda, \leq)$. For each $\lambda\in \Lambda$,
let $L(\lambda)$ be a
simple object and $I(\lambda)$ be its injective envelope.
The costandard object $\nabla(\lambda)$ is
defined as the largest subobject of $I(\lambda)$ whose composition
factors are $L(\mu)$ such that $\mu\leq\lambda$. For example, if $\cal C$
is a highest weight category in the sense of \cite{cps}, then the
costandard objects coincide with the first members of good
filtrations of their injective envelopes. Denote by ${\cal C}_f$ a
full subcategory of $\cal C$ consisting of all finite objects. If
all $\nabla(\lambda)$ belong to ${\cal C}_f$ and ${\cal C}_f$ has
a duality $\tau$, preserving simple objects, then one can define
standard objects in $\cal C$ by
$\Delta(\lambda)=\tau(\nabla(\lambda))$ for $\lambda\in\Lambda$. Any
costandard object $\nabla(\lambda)$ is uniquely defined by the
following universal property: If $W\in {\cal C}$ is such that
$soc(W)=L(\lambda)$ and all other composition factors of $W$ are
$L(\mu)$ for $\mu <\lambda,$ then $W$ is isomorphic to a subobject of
$\nabla(\lambda)$. Symmetrically, if $W/rad W\simeq L(\lambda)$
and all other composition factors of $rad W$ are $L(\mu)$ for $\mu
<\lambda,$ then $W$ is isomorphic to a factor of
$\Delta(\lambda)$.

\section{Hochschild-Serre spectral sequences}
Let $G$ be an affine supergroup and $N$ be a normal subsupergroup
of $G$. The dur $K$-sheaf $\widetilde{\widetilde{G/N}}$ is an affine
supergroup (see Theorem 6.2 of \cite{z1}). Moreover,
$K[\widetilde{\widetilde{G/N}}]\simeq K[G]^N$. To simplify our notation
we denote $\widetilde{\widetilde{G/N}}$ just by $G/N$. If $V\in G-mod$,
then $V^N$ has a natural structure of a $G/N$-supermodule. More
precisely, by Proposition 3.1 from \cite{z}, $V$ is embedded (as a
supercomodule) into a direct sum of several copies of $K[G]$ and
$K[G]^c$. It follows that $V^N$ is the largest subsupermodule of
the $G$-supermodule $V$ whose coefficient space $cf(V^N)$ lies in
$K[G]^N$. Denote the canonical epimorphism $G\to G/N$ by
$\pi$. Following notations of \cite{z}, there is a restriction
functor $\pi_0 : G/N-mod\to G-mod$. A proof of the following
lemma and proposition can be copied from Lemma 6.4 and Proposition 6.6 in Part II of \cite{jan}.
\begin{lm}
The functor $V\to V^N$ is left exact and right adjoint to $\pi_0$.
\end{lm}
\begin{pr}
For $M\in G/N-mod$ and $V, U\in G-mod$ such that $\dim U <\infty,$ there are the
following spectral sequences :

$1) E^{n, m}_2=Ext^n_{G/N}(M, Ext^m_N(U, V))\Rightarrow
Ext^{n+m}_G(M\otimes U, V),$

$2) E^{n, m}_2=Ext^n_{G/N}(M, H^m(N, V))\Rightarrow
Ext^{n+m}_G(M, V),$ and 

$3) E^{n, m}_2=H^n(G/N, H^m(N, V))\Rightarrow H^{n+m}(G, V).$
\end{pr}

\section{Auxiliary results}

Let $G$ be an affine supergroup. A right action of
$G\times G$ on $G$ is given by $(g, (g_1, g_2))\mapsto g_1^{-1}gg_2$ 
for $g,g_1, g_2\in G(A)$ and $A\in SAlg_K$. Its dual morphism $\rho : K[G]\to
K[G]\otimes K[G]^{\otimes 2}$  is defined by $$\rho(f)=
\sum (-1)^{|f_1||f_2|}f_2\otimes s_G(f_1)\otimes f_3$$
for any $f\in K[G]$, where $s_G$ is the antipode of $G$.
Composing this action of $G\times G$ on $G$ with the diagonal inclusion $G\to G\times G$, we obtain a
right adjoint action of $G$ on itself. Its dual morphism $\nu_l$ is given by
$$\nu_l(f)=\sum (-1)^{|f_1||f_2|}f_2\otimes s_G(f_1)f_3,$$
(see \cite{z1}).

Let $L$ and $H$ be affine supergroups, $V$ be a $L$-supermodule via $\rho_V(v)=\sum v_1\otimes f_2$
and $W$ be a $H$-supermodule via $\rho_W(w)=\sum w_1\otimes h_2$.
Then the superspace $V\otimes W$
has a natural structure of a $L\times H$-supermodule given by
$$\rho_{V\otimes W}(v\otimes w)=\sum (-1)^{|w_1||f_2|}v_1\otimes
w_1\otimes f_2\otimes h_2.$$ 
Let $\sigma$ be a Hopf
superalgebra anti-isomorphism of $K[L]$ and $\tau$ be a Hopf
superalgebra anti-isomorphism of $K[H]$. Then 
$\sigma\otimes\tau$ is an anti-isomorphism of Hopf superalgebra
$K[L\times H]$ and there are dualities $V\to V^{<\sigma>}, W\to
W^{<\tau>}$ and $U\to U^{<\sigma\otimes\tau>}$ of the categories
$L-smod, H-smod$ and $L\times H-smod$, repectively (see \cite{z} for more details).

\begin{lm}
If $V$ is a finite-dimensional $L$-supermodule and $W$ is a finite-dimensional $H$-supermodule,
then $V^{<\sigma>}\otimes W^{<\tau>}$ and $(V\otimes
W)^{<\sigma\otimes\tau>}$ are canonically isomorphic as $L\times
H$-supermodules. 
Additionally, assume that $V'$ is a $L$-supermodule and
$W'$ is a $H$-supermodule, and consider $(V\otimes W)$ and
$(V'\otimes W')$ as $L\times H$-supermodules, and 
$(V\otimes V')$ and $(W\otimes W')$ as supermodules
with respect to the diagonal action of $L$ and $H$, respectively.
Then $(V\otimes W)\otimes (V'\otimes W')$ and $(V\otimes V')\otimes (W\otimes W')$ 
are canonically isomorphic as $L\times H$-supermodules.
\end{lm}
\begin{proof}
Let $v_1,\ldots , v_s$ be a ${\bf Z}_2$-homogeneous basis
of $V$ and $w_1,\ldots , w_t$ be a ${\bf Z}_2$-homogeneous basis
of $W$. Denote by $v_1^*\ldots , v_s^*$ and $w_1^*\ldots , w_t^*$
the dual homogeneous bases of $V^*$ and $W^*$, respectively. The
first isomorphism is given by $v_i^*\otimes w_j^*\mapsto
(-1)^{|v_i||w_j|}(v_i\otimes w_j)^*$ for  $1\leq i\leq s$ and $1\leq j\leq t$. 
The second isomorphism is given by 
$$(v\otimes w)\otimes
(v'\otimes w')\mapsto (-1)^{|w||v'|}(v\otimes v')\otimes (w\otimes
w').$$ 
A routine verification of these statements is left to the reader.
\end{proof}
\bigskip

Set $G=GL(m|n)$. In what follows denote by $\tau$ an anti-isomorphism of $K[G]$
defined by $c_{ij}\mapsto (-1)^{|i|(|j|+1)}c_{ji}$ for $1\leq i, j\leq m+n.$ 
Borel subsupergroups of $G$ consisting of lower (or upper) triangular matrices are
denoted by $B^-$ (or $B^+$, respectively). A supergroup
$T=B^-\bigcap B^+$ is a maximal torus in $G$, called the {\it
standard} torus. The category $T$-mod is semisimple, and every simple
$T$-supermodule $V$ is one-dimensional and uniquely defined by its
character $\lambda\in X(T)={\bf Z}^{m+n}$ and by its parity
$\epsilon=0$ or $1$. We denote such a simple $T$-module by $K_{\lambda}^{\epsilon}$. 
Denote induced supermodules $H^0(G/B^- , K_{\lambda}^{\epsilon})$ and
$H^0(G/B^+, K_{\lambda}^{\epsilon})$ by
$H^0_-(\lambda^{\epsilon})$ and $H^0_+(\lambda^{\epsilon})$,
respectively, and Weyl supermodules
$H^0_-(\lambda^{\epsilon})^{<\tau>}$ and
$H^0_+(\lambda^{\epsilon})^{<\tau>}$ by
$V_-(\lambda^{\epsilon})$ and $V_+(\lambda^{\epsilon})$, respectively (all
details can be found in \cite{z}). All arguments in \cite{z} can
be symmetrically restated for $B^+$. In particular, the category
$G-mod$ is a highest weight category with respect to the inverse
dominant order on $X(T)$ (see Remark 5.3 of \cite{z}). 
The corresponding costandard and standard objects are
$H^0_+(\lambda^{\epsilon})$ and $V_+(\lambda^{\epsilon})$,
respectively, where $\epsilon=0$ or $1$ and $\lambda\in X(T)^-$.
Here $H^+(\lambda)=H^0_+(\lambda^0)=H^0_+(\lambda^1)^c,
V_+(\lambda)=V_+(\lambda^0)=V_+(\lambda^1)^c$ and
$$X(T)^- =\{\lambda=(\lambda_1 ,\ldots ,\lambda_{m+n})\in {\bf Z}^{m+n}|
\lambda_1\leq\ldots\leq\lambda_m,
\lambda_{m+1}\leq\ldots\leq\lambda_{m+n}\}.$$ 
The parity
$\epsilon$ coincides with the parity of the one-dimensional
subsuperspace $H^0_+(\lambda^{\epsilon})_{\lambda}$ (or the parity of the one-dimensional subsuperspace
$V_+(\lambda^{\epsilon})_{\lambda}$, respectively). The following lemma is now
obvious.
\begin{lm}
For every $\lambda\in X(T)^+$ and $\epsilon=0$ or $1$ there are natural
isomorphisms $H^0_-(\lambda^{\epsilon})^*\simeq
V_+(-\lambda^{\epsilon})$ and $V_-(\lambda^{\epsilon})^*\simeq
H^0_+(-\lambda^{\epsilon}).$
\end{lm}

It is clear that $$V_+(\lambda^{\epsilon})\otimes
V_-(\mu^{\epsilon'})\simeq V_+(\lambda^{\pi})\otimes
V_-(\mu)\simeq V_+(\lambda)\otimes V_-(\mu^{\pi})$$ and
$$H^0_+(\lambda^{\epsilon})\otimes H^0_-(\mu^{\epsilon'})
\simeq H^0_+(\lambda^{\pi})\otimes H^0_-(\mu)\simeq
H^0_+(\lambda)\otimes H^0_-(\mu^{\pi}),$$ where $\pi \equiv \epsilon
+\epsilon'\pmod 2$.

\begin{pr}
For every $\lambda_1, \mu_1\in X(T)^-, \lambda_2, \mu_2\in X(T)^+$ and
$\epsilon,\epsilon'=0$ or $1$, a superspace $Ext^i_{G\times
G}(V_+(\lambda_1^{\epsilon})\otimes V_-(\lambda_2),
H^0_+(\mu_1^{\epsilon'})\otimes H^0_-(\mu_2))$ is not equal to
zero if and only if $i=0$, $\lambda_1=\mu_1$ and $\lambda_2=\mu_2$. In this
case $Hom_{G\times G}(V_+(\lambda_1^{\epsilon})\otimes
V_-(\lambda_2), H^0_+(\mu_1^{\epsilon'})\otimes H^0_-(\mu_2))$ is
a one-dimensional superspace of parity $\epsilon +\epsilon'\pmod
2$.
\end{pr}
\begin{proof} Consider an exact sequence of algebraic supergroups
$$1\to G\to G\times G\to G\to 1,$$
where the epimorphism $G\times G\to G$ is dual to a monomorphism
$K[G]\to K[G]\otimes K[G]$ defined by $f\mapsto 1\otimes f$. The
kernel of this epimorphism coincides with $G\times 1\simeq G$.
Combining Proposition 3.2 of \cite{z} and Lemma 3.1 we infer
$$Ext^i_{G\times
G}(V_+(\lambda_1^{\epsilon})\otimes V_-(\lambda_2),
H^0_+(\mu_1^{\epsilon'})\otimes H^0_-(\mu_2))\simeq$$$$
H^i(G\times G, V_+(\lambda_1^{\epsilon})^*\otimes
V_-(\lambda_2)^*\otimes H^0_+(\mu_1^{\epsilon'})\otimes
H^0_-(\mu_2))\simeq$$$$H^i(G\times G,
V_+(\lambda_1^{\epsilon})^*\otimes H^0_+(\mu_1^{\epsilon'})\otimes
V_-(\lambda_2)^*\otimes H^0_-(\mu_2)).$$ On the other hand,
$$H^i(G\times 1,
V_+(\lambda_1^{\epsilon})^*\otimes H^0_+(\mu_1^{\epsilon'})\otimes
V_-(\lambda_2)^*\otimes H^0_-(\mu_2))\simeq$$
$$Ext^i_G(V_+(\lambda_1^{\epsilon}),
H^0_+(\mu_1^{\epsilon'}))\otimes (V_-(\lambda_2)^*\otimes
H^0_-(\mu_2))=0$$ whenever $i>0$ (see Theorem 5.5 of \cite{z}).
Proposition 2.1 and standard spectral sequence arguments imply 
$$Ext^i_{G\times
G}(V_+(\lambda_1^{\epsilon})\otimes V_-(\lambda_2),
H^0_+(\mu_1^{\epsilon'})\otimes H^0_-(\mu_2))\simeq$$
$$Hom_G(V_+(\lambda_1^{\epsilon}),
H^0_+(\mu_1^{\epsilon'}))\otimes Ext^i_G(V_-(\lambda_2),
H^0_-(\mu_2)),$$ where in the second tensor multiplier $G$ is
identified with $1\times G$. The claim follows from universal
properties of standard and costandard objects.
\end{proof}
\bigskip

Let $G_1, G_2$ be affine supergroups and let $H_1\leq G_1, H_2\leq
G_2$ be their (closed) subsupergroups.
\begin{lm}(see Lemma 3.8 in Part I of \cite{jan})
If $V_i$ is a $H_i$-supermodule for $i=1, 2$, then there is a
canonical isomorphism $ind^{G_1\times G_2}_{H_1\times
H_2}V_1\otimes V_2\simeq ind^{G_1}_{H_1} V_1\otimes
ind^{G_2}_{H_2} V_2$ of $G_1\times G_2$-supermodules.
\end{lm}
\begin{proof}
Combine Proposition 3.3 of \cite{z} and Lemma 3.1. 
\end{proof}
\bigskip

The character group $X(T\times T)$ is identified with $X(T)\times
X(T)$. It is partially ordered by $(\lambda, \lambda')\leq (\mu,
\mu')$ if and only if $\lambda\geq\lambda'$ and $\mu\leq\mu'$. This ordering
corresponds to the Borel subsupergroup $B^+\times B^-$ of $G\times
G$. The epimorphism $K[B^+\times B^-]\to K[T\times T]$ is split
and $K[T\times T]$ can be canonically identified with a Hopf
subsuperalgebra of $K[B^+\times B^-]$ generated by the group-like
elements $c_{ii}^{\pm 1}\otimes c_{jj}^{\pm 1}$ for $1\leq i, j\leq
m+n$. Denote the kernel of this epimorphism by $J$. Let $V$ be a
$B^+\times B^-$-supermodule, a weight $(\lambda, \lambda')\in
X(T\times T)$ and a ${\bf Z}_2$-homogeneous vector $v\in
V_{(\lambda, \lambda')}$. As in Proposition 5.3 of \cite{z}, we have
$\rho_V(v)=v\otimes c^{(\lambda, \lambda')}+ y$, where $c^{(\lambda,
\lambda')}=\prod_{1\leq i\leq
m+n}c_{ii}^{\lambda_i}\otimes\prod_{1\leq i\leq
m+n}c_{ii}^{\lambda'_i}$ and 
$y\in (\sum_{(\mu, \mu')< (\lambda, \lambda')}
V_{(\mu, \mu')})\otimes J.$
In particular, every simple $B^+\times
B^-$-supermodule is one-dimensional and isomorphic to
$K^{\epsilon}_{(\lambda, \lambda')}\simeq
K^{\epsilon}_{\lambda}\otimes K_{\lambda'}$. If a
$B^+\times B^-$-supermodule is generated by a (${\bf
Z}_2$-homogeneous) vector $v$ of weight $(\lambda, \lambda')$,
then $V_{(\mu, \mu')}\neq 0$ implies $(\mu, \mu')\leq (\lambda,
\lambda')$ and $\dim V_{(\lambda, \lambda')}=1$. Using Proposition
5.4 of \cite{z}, one can easily check that 
if $\pi_{+-} :K[G]\to K[B^+\times B^-]$ and $\pi_{-+} : K[G]\to K[B^-\times
B^+]$ are canonical epimorphisms, then 
the morphism of superalgebras
$$K[G\times G]\stackrel{(\pi_{+-}\otimes\pi_{-+})\delta_{G\times G}}{\longrightarrow} K[B^+\times B^-]
\otimes K[B^-\times B^+]$$ is an inclusion.
Now we are ready to prove the following lemmas.

\begin{lm}
A class of costandard objects in the category of $G\times G$-supermodules consists of supermodules 
$H^0_+(\lambda^{\epsilon})\otimes H^0_-(\lambda')$,
where $(\lambda , \lambda')\in X(T)^-\times X(T)^+$ and $\epsilon=0$ or $1$. 
Moreover, the simple socle of $H^0_+(\lambda^{\epsilon})\otimes H^0_-(\lambda')$ is
isomorphic to $L(-\lambda^{\epsilon})\otimes L(\lambda')$.
\end{lm}
\begin{proof}
A word-by-word repetition of the proofs of Propositions 5.5
and 5.6 of \cite{z}.
\end{proof}
\begin{lm}
A class of standard objects in the category of $G\times G$-supermodules consists of supermodules 
$V_+(\lambda^{\epsilon})\otimes V_-(\lambda')$,
where $(\lambda , \lambda')\in X(T)^-\times X(T)^+$ and $\epsilon=0$ or $1$.
\end{lm}
\begin{proof} It follows from Lemma 3.1.
\end{proof}
\bigskip

Following \cite{z}, we call a $G\times G$-supermodule $V$ {\it
restricted} if and only if for every $(\lambda, \lambda')\in X(T)^-\times
X(T)^+$ a dimension of $Hom_{G\times G}(V_+(\lambda)\otimes V_-(\lambda'), V)$ is finite 
and the set $${\hat V} =\{(\lambda,
\lambda')\in X(T)^-\times X(T)^+ | Hom_{G\times
G}(V_+(\lambda)\otimes V_-(\lambda'), V)\neq 0\}$$ does not
contain any infinite decreasing chain.

\begin{lm}
A restricted $G\times G$-supermodule $V$ has a costandard (or
good) filtration with quotients $H^0_+(\lambda^{\epsilon})\otimes
H^0_-(\lambda')$ if and only if one of the following equivalent conditions hold for every $(\lambda, \lambda')\in X(T)^-\times X(T)^+$:
\begin{enumerate}
\item $Ext^1_{G\times G}(V_+(\lambda)\otimes V_-(\lambda'), V)=0$,
\item $Ext^i_{G\times G}(V_+(\lambda)\otimes V_-(\lambda'), V)=0$
for any $i\geq 1$.
\end{enumerate}
\end{lm}
If $V$ has a good filtration, then the multiplicity of a factor
$H^0_+(\lambda^{\epsilon})\otimes H^0_-(\lambda')$ is equal to
$\dim Hom_{G\times G}(V_+(\lambda^{\epsilon})\otimes
V_-(\lambda'), V)_0$ (see Remark 5.5 in \cite{z}).
\begin{cor}
Every injective $G\times G$-supermodule has a good filtration with
quotients $H^0_+(\lambda^{\epsilon})\otimes H^0_-(\lambda')$. In
particular, $G\times G$-mod is a highest weight category with
poset $(X(T)^-\times\{0, 1\})\times X(T)^+$ ordered as above.
\end{cor}
\begin{rem}
Let $\delta=(\delta_1,\ldots ,\delta_t)$ be an element of the set
$\{+, - \}^t$ and $G^t=\underbrace{G\times\ldots\times G}_t$.
Extending the above arguments, one can prove that $G^t$-mod is a
highest weight category with respect to a poset $(X(T)^{\delta_1}\times\{0,
1\})\times\ldots\times X(T)^{\delta_t}$; costandard and
standard objects in $G^t$-mod are
$H^0_{\delta_1}(\lambda_1^{\epsilon})\otimes\ldots\otimes
H^0_{\delta_t}(\lambda_t)$ and
$V_{\delta_1}(\lambda_1^{\epsilon})\otimes\ldots\otimes
V_{\delta_t}(\lambda_t)$, respectively.
\end{rem}

\begin{pr}(see Proposition 4.20 in Part II of \cite{jan})
The superalgebra $K[G]$, considered as a $G\times G$-supermodule via
$\rho$, satisfies the following conditions:
\begin{enumerate}
\item $Ext^i_{G\times G}(V_+(\lambda)\otimes V_-(\lambda'), K[G])=0$
for all $i\geq 1$ and $(\lambda, \lambda')\in X(T)^-\times X(T)^+$,
\item $Hom_{G\times G}(V_+(\lambda^{\epsilon})\otimes V_-(\lambda'), K[G])_0\neq 0$
if and only if $\epsilon=0$ and $\lambda=-\lambda'$.
\end{enumerate}
Additionally, $\dim Hom_{G\times G}(V_+(\lambda)\otimes V_-(-\lambda),
K[G])_0 =1$.
\end{pr}
\begin{proof} By Proposition 2.1, there is a spectral sequence
$$H^n(1\times G, H^m(G\times 1, H^0_-(-\lambda^{\epsilon})\otimes H^0_+(-\lambda')\otimes K[G]))\Rightarrow
H^{m+n}(G\times G, H^0_-(-\lambda^{\epsilon})\otimes
H^0_+(-\lambda')\otimes K[G]).$$ On the left-hand side, $K[G]$ is
isomorphic to $K[G]_l$ as a $G=G\times 1$-supermodule. Since
$K[G]_l\simeq K[G]$ is injective (see \cite{z}), the spectral sequence
degenerates and yields an isomorphism
$$H^n(1\times G, (H^0_-(-\lambda^{\epsilon})\otimes K[G]_l)^G\otimes
H^0_+(-\lambda'))\simeq Ext^n_{G\times G}(
V_+(\lambda^{\epsilon})\otimes V_-(\lambda'), K[G]).$$ Finally,
$(H^0_-(-\lambda^{\epsilon})\otimes K[G]_l)^G\simeq Ind^G_G
H^0_-(-\lambda^{\epsilon})=H^0_-(-\lambda^{\epsilon})$ and the
space on left is isomorphic to
$$H^n(G, H^0_-(-\lambda^{\epsilon})\otimes H^0_+(-\lambda'))\simeq
Ext^n_G(V_+(\lambda^{\epsilon}), H^0_+(-\lambda')).$$ Thus
$Ext^n_{G\times G}( V_+(\lambda^{\epsilon})\otimes V_-(\lambda'),
K[G])\neq 0$ if and only if $n=0$ and $\lambda=-\lambda'$. In this 
case, $Hom_{G\times G}( V_+(\lambda^{\epsilon})\otimes
V_-(-\lambda), K[G])$ is one-dimensional. This module is even if and only if
$\epsilon=0$.
\end{proof}

\section{Some results about filtrations}

Let $\cal C$ be a highest weight category from Section 1, with a
duality $\tau$. We call $\tau$ a {\em Chevalley duality}. In what
follows we assume that all costandard objects are finite and {\em
Schurian}, that is $End_{\cal C}(\nabla(\lambda))=K$ for every
$\lambda\in\Lambda$. For $\lambda\in\Lambda$, denote by
$(\lambda]$ the (possibly infinite) interval $\{\mu\in\Lambda
|\mu\leq\lambda\}$ and denote by $(\lambda)$ the open interval $\{\mu |\mu <\lambda\}$.
Let $\Gamma\subseteq\Lambda$ and $M\in\cal
C$. We say that $M$ belongs to $\Gamma$ if and only if every composition factor $L(\lambda)$
of $M$ satisfies $\lambda\in\Gamma$. The largest subobject of $N\in \cal C$ that belongs to $\Gamma$ will be denoted
by $O_{\Gamma}(N)$. Symmetrically, $N$ contains a unique minimal
subobject $O^{\Gamma}(N)$ such that $N/O^{\Gamma}(N)$ belongs to
$\Gamma$. For example, $\nabla(\lambda)=O_{(\lambda]}(I(\lambda))$ for 
$\lambda\in\Lambda$. Finally, denote by $[M : L(\lambda)]$ the
supremum of multiplicities of a simple object $L(\lambda)$ in
composition series of all finite subobjects of $M$.

A full subcategory consisting of all objects $M$ with
$O_{\Gamma}(M)=M$ will be denoted by ${\cal C}[\Gamma]$. It is obvious
that $O_{\Gamma}$ is a left exact functor from $\cal C$ to ${\cal
C}[\Gamma]$ and it commutes with direct sums. The functor
$O^{\Gamma} : {\cal C}\to {\cal C}$ also commutes with direct sums,
but it is not right exact in general. In fact, for every exact
sequence $$0\to X\to Y\to Z\to 0,$$ the map $O^{\Gamma}(Y)\to
O^{\Gamma}(Z)$ is an epimorphism and $O^{\Gamma}(X)\subseteq
X\bigcap O^{\Gamma}(Y)$. Nevertheless, it is possible that
$O^{\Gamma}(X)$ is a proper subobject of $X\bigcap O^{\Gamma}(Y)$.

A subset $\Gamma\subseteq\Lambda$ is called an {\it ideal}, if
$\lambda\in \Gamma$ implies $(\lambda]\subseteq\Gamma$. 
If $\Gamma=\bigcup_{1\leq k\leq s}(\lambda_k]$, then 
we say that $\Gamma$ is finitely generated (by the elements
$\lambda_1,\ldots ,\lambda_s$). We assume that every finitely
generated ideal $\Gamma\subseteq\Lambda$ has a decreasing chain of
finitely generated subideals
$\Gamma=\Gamma_0\supseteq\Gamma_1\supseteq\Gamma_2\supseteq\ldots$
such that $\Gamma\setminus\Gamma_k$ is a finite set for all $k\geq
0$ and $\bigcap_{k\geq 0}\Gamma_k=\emptyset$. For example, 
if every $\lambda\in\Lambda$ has finitely many predecessors 
($\mu<\lambda$ such that there is no $\pi\in\Lambda$ between $\lambda$
and $\mu$), then this assumption is satisfied. In this case, for $\Gamma=\bigcup_{1\leq k\leq s}(\lambda_k]$ 
we can put $\Gamma_1=\bigcup_{1\leq k\leq s}(\lambda_k)$ and observe that
$\Gamma\setminus\Gamma_1$ is finite, and $\Gamma_1$ is generated by all predecessors of the elements
$\lambda_1,\ldots ,\lambda_s$.
 
From now on, $\Gamma$ is a finitely
generated ideal, unless stated otherwise. If
$\Gamma=\bigcup_{1\leq j\leq k}(\pi_j]$ and $\lambda\in\Gamma$, then
denote by $(\lambda, \Gamma]$ (and $[\lambda, \Gamma]$, respectively)
the finite set $\bigcup_{1\leq j\leq k}(\lambda , \pi_j]$
(and the finite set $\bigcup_{1\leq j\leq k}[\lambda ,
\pi_j]$, respectively).  

\begin{ex}
The category $G-mod$ satisfies all of the above conditions. In fact,
any $\lambda\in X(T)^+$ has only finitely many predecessors. More precisely, if $\mu$ is a predecessor of $\lambda$ and
$\sum_{1\leq i\leq m}\mu_i < \sum_{1\leq i\leq m}\lambda_i$, then
$\mu\leq \lambda' <\lambda$, where
$\lambda'=(\lambda_1,\ldots,\lambda_{m-1},\lambda_m -1
|\lambda_{m+1}+1,\lambda_{m+2},\ldots,\lambda_{m+n})$. Therefore
either $\mu=\lambda'$ or $\mu_+\leq\lambda_+,
\mu_-\leq\lambda_- $. It remains to observe that for any (ordered)
partition $\pi$, there are only finitely many partitions of the
same length, smaller or equal to $\pi$. Repeating these arguments as
many times as needed, one can prove that $G^t$-mod satisfies all
the above conditions for any root data $(X(T)^{\delta_1}\times
\{0, 1\})\times\ldots\times X(T)^{\delta_t}$. 
\end{ex}
According to \cite{cps}, the subcategory ${\cal C}[\Gamma]$ is a highest weight category with
costandard objects $\nabla(\lambda)$ and finite injective envelopes
$I_{\Gamma}(\lambda)=O_{\Gamma}(I(\lambda))$ for every $\lambda\in\Gamma$. By Theorem 3.9 of \cite{cps}, 
for every $M, N\in {\cal C}[\Gamma]$ we have 
$Ext^i_{\cal C}(M, N)= Ext^i_{{\cal C}[\Gamma]}(M, N)$.

Let $\cal R$ be a class of objects from $\cal C$. An increasing
filtration
$$0=M_0\subseteq M_1\subseteq M_2\subseteq\ldots$$
of an object $M$ such that $M_i/M_{i-1}\in {\cal R}$ for every $i\geq 1,$ and
$\bigcup_{i\geq 1}M_i=M$, is called an increasing ${\cal
R}$-filtration. If $M$ has a decreasing filtration
$$M=M_0\supseteq M_1\supseteq M_2\supseteq\ldots$$
such that $M_i/M_{i+1}\in {\cal R}$ for every $i\geq 0,$ and $\bigcap_{i\geq
0}M_i=0$, then it is called a decreasing ${\cal R}$-filtration. For
example, every injective envelope $I(\lambda)$ has an increasing
$\nabla$-filtration, where $\nabla= \{\nabla(\lambda)
|\lambda\in\Lambda\}$. To finalize our notations, let us denote the
class $\{\Delta(\lambda) |\lambda\in\Lambda\}$ by $\Delta$. Recall
that $L(\lambda)=\Delta(\lambda)/rad \Delta(\lambda)$ and all other
composition factors $L(\mu)$ of $\Delta(\lambda)$ satisfy $\mu
<\lambda$.
\begin{lm}
If $M$ belongs to $\Gamma$ and $Ext^i_{\cal C}(M,\nabla(\lambda))\neq 0$ for some
$i>0$, then there is a composition factor $L(\mu)$ of $M$ such
that $\mu >\lambda$.
\end{lm}
\begin{proof} Without a loss of generality one can suppose that
$\lambda\in\Gamma$. Consider a short exact sequence
$$0\to\nabla(\lambda)\to I_{\Gamma}(\lambda)\to Q\to 0 ,$$ where $Q$ has a finite $\nabla$-filtration with quotients
$\nabla(\mu)$, where $\mu\in(\lambda, \Gamma]$. A fragment of a long exact sequence
$$\ldots \to Ext^{i-1}_{\cal C}(M, Q)\to Ext^i_{\cal C}(M, \nabla(\lambda))\to 0$$
shows that $Ext^{i-1}_{\cal C}(M, Q)\neq 0$. In particular, $Ext^{i-1}_{\cal C}(M, \nabla(\mu))$ is nonzero for a subquotient $\nabla(\mu)$ of 
a $\nabla$-filtration of $Q$. We use an induction on $i$ to 
show that $Hom_{\cal C}(M, \nabla(\mu))\neq 0$ for some $\mu
>\lambda$.
\end{proof}
\bigskip

For every $\lambda\in \Lambda$, denote $P_{\Gamma}(\lambda)=\tau(I_{\Gamma}(\lambda))$. 
Then $P_{\Gamma}(\lambda)$ is a projective cover of $L(\lambda)$ in ${\cal C}[\Gamma]$ and
$P_{\Gamma}(\lambda)$ has a $\Delta$-filtration which is Chevalley dual to the corresponding
$\nabla$-filtration of $I_{\Gamma}(\lambda)$ for $\lambda\in\Gamma$. The
following lemma is a symmetric variant of Lemma 4.1.
\begin{lm}
Let $M\in {\cal C}[\Gamma]$. If $Ext^i_{\cal C}(\Delta(\lambda),
M)\neq 0$ for some $i
> 0$, then there is a composition factor $L(\mu)$ of $M$ such that $\mu
>\lambda$.
\end{lm}
\begin{cor}(compare with Theorem 3.11 of \cite{cps})
For every $\lambda, \mu\in\Lambda$ and $i > 0$, we have $Ext^i_{\cal
C}(\Delta(\lambda), \nabla(\mu))=0$.
\end{cor}
\begin{cor}
Let $M\in {\cal C}$. If $Ext^1_{\cal C}(\Delta(\lambda), M)=0$
(or $Ext^1_{\cal C}(M, \nabla(\lambda))=0$, respectively) for all
$\lambda\in\Lambda$, then 
$Ext^1_{\cal C}(\Delta(\lambda), O_{\Gamma}(M))=0$
(or $Ext^1_{\cal C}(M/O^{\Gamma}(M), \nabla(\lambda))=0$, respectively)
for all $\lambda\in\Lambda$.
\end{cor}
\begin{proof} If $Ext^1_{\cal C}(\Delta(\lambda), O_{\Gamma}(M))\neq 0$,
then $\lambda\in\Gamma$. It remains to consider the following
fragment of a long exact sequence
$$Hom_{\cal C}(\Delta(\lambda), M/O_{\Gamma}(M))\to
Ext^1_{\cal C}(\Delta(\lambda), O_{\Gamma}(M))\to 0$$ and observe
that $Hom_{\cal C}(\Delta(\lambda), M/O_{\Gamma}(M))=0$.
The proof of the second statement is analogous.
\end{proof} 
\bigskip

An object $M\in {\cal C}[\Gamma]$ is called $\Gamma$-restricted if and only if
$[M : L(\lambda)]$ is finite for every $\lambda\in\Gamma$. 
\begin{lm}
 Let $M$ be a $\Gamma$-restricted object such that
$Ext^1_{\cal C}(\Delta(\lambda), M)=0$ (or $Ext^1_{\cal
C}(M, \nabla(\lambda))=0$, respectively) for every $\lambda\in\Lambda$. Then for every
$\lambda\in\Lambda$ and $i > 1$ we have $Ext^i_{\cal
C}(\Delta(\lambda), M)=0$ (or $Ext^i_{\cal C}(M,
\nabla(\lambda))=0$, respectively).
\end{lm}
\begin{proof} As in Lemma 4.1, one can assume that $\lambda\in\Gamma$ and work in the
category ${\cal C}[\Gamma]$. A short exact sequence
$$0\to Q\to P_{\Gamma}(\lambda)\to \Delta(\lambda)\to 0$$
induces
$$\ldots\to Ext^{i-1}_{\cal C}(Q, M)\to Ext^i_{\cal C}(\Delta(\lambda), M)\to 0 .$$
Since $Q$ has a finite $\Delta$-filtration with factors $\Delta(\mu)$,
where $\mu\in (\lambda, \Gamma]$, one can argue by an induction on $i$. The proof of the second statement is similar.
\end{proof}

\begin{ter}
Let $M$ be a $\Gamma$-restricted object. If $Ext^1_{\cal
C}(\Delta(\lambda), M)=0$ for all $\lambda\in\Lambda$, then $M$ has
a decreasing $\nabla$-filtration. If $Ext^1_{\cal C}(M, \nabla(\lambda))=0$ for all
$\lambda\in\Lambda$, then $M$ has an increasing $\Delta$-filtration.
\end{ter}
\begin{proof} Suppose that $Ext^1_{\cal C}(\Delta(\lambda), M)=0$ for all
$\lambda\in\Lambda$. By our assumption there is a decreasing chain
of finitely generated ideals
$$\Gamma=\Gamma_0\supseteq\Gamma_1\supseteq\ldots$$
such that $\Gamma\setminus\Gamma_k$ is finite for every $k\geq 0$ and
$\bigcap_{k\geq 0}\Gamma_k =\emptyset$. For short, 
denote $O_{\Gamma_k}(M)$ just by $M_k$. There is a decreasing chain of
subobjects
$$M=M_0\supseteq M_1\supseteq M_2\supseteq\ldots ,$$
where any quotient $M/M_k$ is finite and $\bigcap_{1\leq k}M_k=0$.
In fact, the socle of any $M/M_k$ belongs to
$\Gamma\setminus\Gamma_k$, and therefore it is finite. Thus $M/M_k$
can be embedded into a finite sum of finite indecomposable
injectives from ${\cal C}[\Gamma]$. Consider the following fragment
of a long exact sequence
$$Hom_{\cal C}(\Delta(\mu), M/M_k)\to Ext^1_{\cal C}(\Delta(\mu), M_k)\to 0\to $$
$$\to Ext^1_{\cal C}(\Delta(\mu), M/M_k)\to Ext^2_{\cal C}(\Delta(\mu), M_k) .$$
If  $Ext^1_{\cal C}(\Delta(\mu), M_k)\neq 0$, then $\mu\in\Gamma_k$.
On the other hand, since the socle of $M/M_k$ does not belong to
$\Gamma_k$, we infer that $Hom_{\cal C}(\Delta(\mu), M/M_k)=0$. In
particular, $Ext^1_{\cal C}(\Delta(\mu), M_k)=0$ for every $\mu$. By
Lemma 4.3,  $Ext^2_{\cal C}(\Delta(\mu), M_k)=0$ for every $\mu$
and therefore, $Ext^1_{\cal C}(\Delta(\mu), M/M_k)=0$. Since
$M_{k-1}/M_k=O_{\Gamma_{k-1}}(M/M_k)$, we can repeat the above
arguments to obtain $Ext^1_{\cal C}(\Delta(\mu),
M_{k-1}/M_k)=0$ for every $\mu$. Finally, any object $M_{k-1}/M_k$ is
finite and we conclude the proof of the first statement by standard arguments from
\cite{jan, dr}.

For the second statement, it is enough to prove that all subobjects
$O^{\Gamma_k}(M)$ are finite. In fact, $O^{\Gamma_k}(M)$ contains a
finite subobject $N$ such that $[N : L(\mu)]=[O^{\Gamma_k}(M) :
L(\mu)]$ for all $\mu\in\Gamma\setminus\Gamma_k$. In particular,
$O^{\Gamma_k}(M)/N$ belongs to $\Gamma_k$, that is
$N=O^{\Gamma_k}(M)$. The final argument is the same as above.
\end{proof}

\section{Donkin-Koppinen filtration and (co)adjoint action invariants}

\begin{lm}
The superalgebra $K[G]$ is a $\Lambda$-restricted $G\times
G$-supermodule, where $\Lambda=(X(T)^-\times\{0, 1\})\times X(T)^+$.
\end{lm}
\begin{proof} It is enough to show that $O_{\Gamma}(K[G])$ is
$\Gamma$-restricted for any (finitely) generated ideal $\Gamma$ of $\Lambda$. Using 
Corollary 4.2 and Proposition 3.2 we
have
$$[O_{\Gamma}(K[G]) : L(-\lambda^{\epsilon})\otimes
L(\lambda')]=\dim Hom_{G\times G}(P_{\Gamma}(-\lambda^{\epsilon},
\lambda'), O_{\Gamma}(K[G]))\leq$$
$$\sum_{(\mu^{\epsilon},
\mu')\in [(\lambda^{\epsilon}, \lambda'), \Gamma]}(P_{\Gamma}(-\lambda^{\epsilon},
\lambda') : V_+(\mu^{\epsilon})\otimes V_-(\mu')) <\infty,
$$ 
where $(P_{\Gamma}(-\lambda^{\epsilon},
\lambda') : V_+(\mu^{\epsilon})\otimes V_-(\mu'))$ is the multiplicity of $V_+(\mu^{\epsilon})\otimes V_-(\mu')$ in a
(finite) $\Delta$-filtration of $P_{\Gamma}(-\lambda^{\epsilon},
\lambda')$.
\end{proof}
\begin{ter}
For every ideal $\Gamma$ of $\Lambda$, the $G\times G$-subsupermodule $O_{\Gamma}(K[G])$
has a decreasing (good) filtration $O_{\Gamma}(K[G])=V_0\supseteq
V_1\supseteq V_2\supseteq\ldots$ such that $V_k/V_{k+1}\simeq
V_-(\lambda_k)^*\otimes H_-^0(\lambda_k)$ for $k\geq 0$. Moreover, for
any pair of indexes $k < l$, either $\lambda_k >\lambda_l$ or $\lambda_k$ and $\lambda_l$ are not comparable.
\end{ter}
\begin{proof} Combine Theorem 4.1 and Lemma 5.1. \end{proof}
\begin{rem}
By Proposition 3.2, a supermodule $V_-(\lambda)^*\otimes
H_-^0(\lambda)$ appears as a factor of some good filtration of
$O_{\Gamma}(K[G])$ if and only if $(-\lambda ,\lambda)\in\Gamma$ and
$(O_{\Gamma}(K[G]) : V_-(\lambda)^*\otimes H_-^0(\lambda))=1$.
\end{rem}
The above filtration will be called a {\it
Donkin-Koppinen} filtration (see \cite{kop}).
\begin{rem}
We believe that Theorem 5.1 can be proved for other classical supergroups. All we need to do  
is to extend the theory from \cite{z} for them. The case of ortho-symplectic supergroups seems to be most promissing, 
but in this case it is not even known whether the corresponding categories of (rational) supermodules are highest weight categories.   
\end{rem}

Consider the coadjoint action $\nu_l$ of $G$ on $K[G]$ and define
a subsuperalgebra $R=K[G]^G$ of coadjoint (rational) invariants. By
definition, $R=\{f\in K[G]|\nu_l(f)=f\otimes 1\}$. Since $\nu_l$ is
dual to the adjoint action of $G$ on itself, these invariants can
be also called adjoint. Moreover, a rational function $f\in K[G]$
belongs to $R$ if and only if $f(g_1^{-1}g_2g_1)=f(g_1)$ for every $A\in SAlg_K, g_1$ and $g_2\in G(A)$.
In other words, invariants from $R$ are {\it absolute} invariants.

A subalgebra of $R$, generated by polynomial invariants, will be
denoted by $R_{pol}$.

Let $H$ be an algebraic supergroup and $V$ be a finite-dimensional $H$-supermodule. Fix a ${\bf Z}_2$-homogeneous basis of $V$, say
$v_1,\ldots , v_p, v_{p+1},\ldots, v_{p+q}$, where $|v_i|=0$ if
$1\leq i\leq p$ and $|v_i|=1$ otherwise. Set $\rho_V(v_i)=\sum_{1\leq
j\leq p+q} v_j\otimes f_{ji}$ for $1\leq i\leq p+q$ and denote
$Tr(\rho)$ (or by $Tr(V)$, see \cite{kv}) the supertrace
$\sum_{1\leq i\leq p}f_{ii}-\sum_{p+1\leq i\leq p+q}f_{ii}$. It is
well known that $Tr(\rho)$ does not depend on a choice of ${\bf
Z}_2$-homogeneous basis of $V$, see \cite{kv, ber, ktr}.
\begin{lm}
The supertrace $Tr(\rho)$ belongs to the superalgebra of
(co)adjoint invariants $K[H]^H$.
\end{lm}
\begin{proof} Denote by $S$ a Hopf subsuperalgebra of $K[H]$ generated
by the elements $f_{ij}$. There is a natural Hopf superalgebra
epimorphism $K[GL(p|q)]\to S$ induced by the map $c_{ij}\mapsto
f_{ij}$. It remains to refer to \cite{ber}.
\end{proof}
\bigskip

Consider $K[G]$ as a $T$-supermodule. Observe that the action of $T$ is a
restriction of the coaction $\rho$ on the subsupergroup $1\times
T\subseteq T\times T$. Recall that a $G$-supermodule $V$ is
semisimple as a $T$-supermodule and $V=\bigoplus_{\lambda\in X(T)}
V_{\lambda}$, where $V_{\lambda}$ is a direct sum of one-dimensional 
$T$-supermodules of weight $\lambda$. In particular,
any element $f\in K[G]$ can be represented as a sum
$\sum_{\lambda\in X(T)} f_{\lambda}$, where $f_{\lambda}\in
K[G]_{\lambda}$. A non-zero summand $f_{\lambda}$ of $f$ will be called
{\it leading} if $\lambda$ is maximal among all $\mu$ such that
$f_{\mu}\neq 0$. The corresponding weight $\lambda$ will be also called
leading.
\begin{lm}
Let $V$ be a $G$-supermodule. Let $v_1,\ldots , v_t$ be a ${\bf Z}_2$-homogeneous
basis of $V$ such that $v_i\in V_{\lambda(i)}$
(it is possible that $\lambda(i)=\lambda(j)$ for $i\neq j$). If
$\rho_V(v_i)=\sum_{1\leq j\leq t}v_j\otimes d_{ji}$ for $1\leq i\leq
t$, then $d_{ji}\in K[G]_{\lambda(i)}$ for $1\leq i, j\leq t$.
\end{lm}
\begin{proof} The map $\rho_V : V\to V_{triv}\otimes K[G]|_T$, where $V_{triv}$ is considered as a
trivial $T$-supermodule, is a morphism of $T$-supermodules.
\end{proof}

\begin{cor}
If $\lambda=\lambda(i)$ is a maximal (or largest, respectively) element, then
$Tr(V)_{\lambda}$ is a leading (or a unique leading, respectively)
summand of $Tr(V)$.
\end{cor}

Let $E$ be a standard $G$-supermodule with ${\bf Z}_2$-homogeneous
basis $e_1,\ldots , e_{m+n}$, where $|e_i|=0$ if $1\leq i\leq m$ and
$|e_i|=1$ otherwise, and such that $\rho_E(e_i)=\sum_{1\leq j\leq
m+n} e_j\otimes c_{ji}$ for $1\leq i\leq m+n$. Denote by $Ber(E)$ a
one-dimensional $G$-supermodule corresponding to a group-like
element (Berezinian)
$Ber(C)=\det(C_{00}-C_{01}C_{11}^{-1}C_{10})\det(C_{11})^{-1}$. It
is clear that $Ber(C)\in K[G]_{\theta}$, where $\theta=(1^m |
(-1)^n)=(\underbrace{1,\ldots , 1}_m | \underbrace{-1,\ldots ,
-1}_n)$.

Let ${\cal I}(r)$ be a set of all maps $I :
\underline{r}\to\underline{m+n}$.  One can consider $I\in
{\cal I}(r)$ as a multi-index $(i_1,\ldots, i_r)$, where
$i_k=I(k)$ for $1\leq k\leq r$. For $I, J\in {\cal I}(r)$ we set
$x_{IJ}=(-1)^{s(I, J)}c_{IJ}$, where $c_{IJ}=\prod_{1\leq k\leq r}
c_{i_k j_k}, s(I, J)=\sum_t |i_t|(\sum_{s <t}|i_s|+|j_s|)$ and
$|i|=|e_i|$. The superspace $E^{\otimes r}$ has a basis
$e_I=e_{i_1}\otimes\ldots\otimes e_{i_r}$ for $I\in {\cal I}(r)$. The superspaces $E^{\otimes r}$,
exterior powers $\Lambda^r(E)$ and symmetric powers $S^r(E)$ have a natural structure of a
$G$-supermodules, see \cite{zm}.
\begin{lm}
The structure of $G$-supermodule on $E^{\otimes r}$ is given by
$$\rho_{E^{\otimes r}}(e_I) =\sum_{J\in {\cal I}(r)} e_{J}\otimes x_{JI}.$$
\end{lm}
\begin{proof} Straightforward calculations.
\end{proof}
\bigskip

Denote by ${\cal LI}(r)$ a subset of ${\cal I}(r)$ consisting
of all multi-indexes $I$ such that for some $k\leq r$ there is $i_1
<\ldots  < i_k\leq m < j_{k+1}\leq\ldots\leq j_r$. Analogously,
denote by ${\cal SI}(r)$ a subset of ${\cal I}(r)$ consisting
of all multi-indexes $I$ such that for some $k\leq r$ there is $i_1
\leq\ldots\leq i_k\leq m < j_{k+1} < \ldots < j_r$. Let $\pi_1 :
E^{\otimes r}\to \Lambda^r(E)$ and $\pi_2 : E^{\otimes r}\to
S^r(E)$ be the canonical epimorphisms. Then $\pi_1(e_I)$
(or $\pi_2(e_I)$, respectively) form a basis of $\Lambda^r(E)$ (or $S^r(E)$, respectively), 
where $I$ runs over ${\cal LI}(r)$ (or ${\cal SI}(r)$, respectively). Recall that the symmetric group $S_r$ acts
on $E^{\otimes r}$ as $e_I\sigma=(-1)^{s(I, \sigma)}e_{I\sigma}$,
where $$s(I, \sigma)=|\{(k, l)| 1\leq k < l\leq r, \
\sigma(k)>\sigma(l), \ |i_{\sigma(k)}|=|i_{\sigma(l)}|=1\}|,$$ see
\cite{mu}. If we replace $E$ by $E^c$ but keep our previous notation,
then this action turns into $e_I\sigma=(-1)^{s'(I,
\sigma)}e_{I\sigma}$, where
$$s'(I, \sigma)=|\{(k, l)| 1\leq k < l\leq r, \
\sigma(k)>\sigma(l), \ |i_{\sigma(k)}|=|i_{\sigma(l)}|=0\}|.$$ Following
\cite{mu}, one can define two actions of $S_r$ on the set of
monomials $x_{IJ}$. Namely, $x_{IJ\star\sigma}=(-1)^{s(J,
\sigma)}x_{IJ\sigma}$ (or $x_{I\star\sigma, J}=(-1)^{s(I,
\sigma)}x_{I\sigma, J}$) and $x_{IJ\circ\sigma}=(-1)^{s'(J,
\sigma)}x_{IJ\sigma}$ (or $x_{I\circ\sigma, J}=(-1)^{s'(I,
\sigma)}x_{I\sigma, J}$). It can be checked easily that
$x_{I\star\sigma , J\star\sigma}=x_{IJ}$ and $x_{I\circ\sigma,
J\circ\sigma}=x_{IJ}$.

\begin{lm}
For any $r\geq 0$, we have $$Tr(\Lambda^r(E))=\sum_{I\in {\cal
LI}(r)}(-1)^{|I|}\sum_{\sigma\in Stab(I)\backslash S_r}
x_{I\circ\sigma, I}$$ and
$$Tr(S^r(E))=\sum_{I\in {\cal
SI}(r)}(-1)^{|I|}\sum_{\sigma\in Stab(I)\backslash S_r}
x_{I\star\sigma, I},$$ where $|I|=\sum_{1\leq k\leq
m+n}|i_k|=|e_I|$.
\end{lm}
\begin{proof} We consider only the first equality, the second one is
symmetrical. It is clear that for any $I\in {\cal LI}(r)$ the
vector $\pi_1(e_I)$ appears in $\rho(\pi_1(e_I))$ only in terms
$$\sum_{\sigma\in Stab(I)\backslash S_r}
\pi_1(e_{I\sigma})\otimes x_{I\sigma, I}.$$ It remains to observe
that $\pi_1(e_{I\sigma})=(-1)^{s'(I, \sigma)}\pi_1(e_I)$.
\end{proof}
\begin{cor}
An invariant $C_r=Tr(\Lambda^r(E))$ has a unique leading summand which has
a weight $(1^r , 0^{m-r} |0^n)$ if $r\leq m$, and $(1^m| r-m, 0^{n-1})$ otherwise.
\end{cor}

Consider a dual $G$-supermodule $E^*$. By definition,
$\rho(e_i^*)=\sum_{1\leq j\leq m+n} e^*_j\otimes c^*_{ji},$ where
$e_i^*(e_k)=\delta_{ik}$ and $c^*_{ji}=(-1)^{|j|(|i|+|j|)}s_G(c_{ij})$ for 
$1\leq i, k\leq m+n$, see \cite{z}. Denote $Tr(S^r(E^*))$ by $D_r$.
\begin{cor}
The invariant $D_r$ has a unique leading summand which is of weight
$(0^m|0^{n-r}, (-1)^r)$ if $r\leq n$ and $(0^{m-1}, n-r|(-1)^n)$ otherwise.
\end{cor}
\begin{proof} Every $e^*_i$ has a weight $(0,\ldots , 0,
\underbrace{-1}_{i-\mbox{th place}}, 0,\ldots , 0)$. Since the
lowest weight appearing in $S^r(E)$ is $(0^m|0^{n-r}, 1^r)$ if $r\leq n$ and $(0^{m-1}, r-n| 1^n)$ otherwise, 
the claim follows by Corollary 5.1.
\end{proof}
\bigskip

Let $f$ be an invariant from $R$. Denote by $M$ a $G\times
G$-subsupermodule generated by $f$. Since $M$ is finite-dimensional,
the ideal $$\Gamma=\bigcup_{(\mu^{\epsilon},\mu'), [M :
L(-\mu^{\epsilon})\otimes L(\mu')]\neq 0} ((\mu^{\epsilon},\mu')]$$
is finitely generated. Fix a Donkin-Koppinen filtration of
$O_{\Gamma}(K[G])$, say $O_{\Gamma}(K[G])=V_0\supseteq V_1\supseteq
V_2\supseteq\ldots$ as in Theorem 5.1.
\begin{lm}
For every $t\geq 0$, a superspace $(V_t/V_{t+1})^G$ is even and one-dimensional. Moreover, a non-zero basic vector from
$(V_t/V_{t+1})^G$ has a unique leading summand of weight
$\lambda_t$.
\end{lm}
\begin{proof} A non-zero basic vector $g$ of $(V_t/V_{t+1})^G\simeq
Hom_G(V_-(\lambda_t), H^0_-(\lambda_t))$ can be represented as a sum
$\sum_{\mu\leq\lambda_t}\sum_{1\leq i\leq k_{\mu}}\phi_{i,
\mu}\otimes v_{i, \mu},$ where $\phi_{i, \mu}\in V_-(\lambda_t)^*$
and  $v_{i, \mu}$ runs over a basis of $H^0_-(\lambda_t)_{\mu}$.
Additionaly, $g_{\mu}=\sum_{1\leq i\leq k_{\mu}}\phi_{i, \mu}\otimes
v_{i, \mu}$ with respect to the right $T$-action, $k_{\lambda_t}=1$
and $\phi_{1, \lambda_t}(V_-(\lambda_t))\neq 0$.
\end{proof}

\begin{cor}
The superalgebra $R$ is even.
\end{cor}
Let $\lambda=(\lambda_1,\ldots ,\lambda_m |\lambda_{m+1},\ldots
,\lambda_{m+n})\in X(T)^+$. As in \cite{z}, we denote
$(\lambda_1,\ldots ,\lambda_m)$ and $(\lambda_{m+1},\ldots
,\lambda_{m+n})$ by $\lambda_+$ and $\lambda_-$, respectively. An
invariant
$$f_{\lambda}=Ber(C)^{\lambda_{m+n}}\prod_{1\leq s\leq
n-1}(Ber(C)D_s)^{\lambda_{m+n-s}-\lambda_{m+n-s+1}}C_m^{|\lambda_-|+\lambda_m}\prod_{1\leq
t\leq m-1}C_t^{\lambda_t -\lambda_{t+1}}$$ has a unique leading
summand of weight $\lambda$.

If $k$ is sufficiently large, then $M\bigcap V_{k+1}=0$. In
particular, all weights of $f$ are among the weights of
$\bigoplus_{0\leq t\leq k} V_t/V_{t+1}$. Assume that $f\in
V_t\setminus V_{t+1}$ for $t\leq k$. By Lemma 5.5, $f$ has a leading
summand of weight $\lambda_t$. By Corollary 4.2 and Remark 5.1, the
invariant $f_{\lambda_t}$ also belongs to $V_t\setminus
V_{t+1}$. It follows that the leading summand of $f_{\lambda_t}$
coincides with the above leading summand of $f$ up to a non-zero
scalar. We will call $f* =f_{\lambda_t}$ a {\it companion} of $f$.

Denote by $A$ a free Laurent polynomial algebra $K[x_1^{\pm
1},\ldots , x_m^{\pm 1}, y_1^{\pm 1},\ldots , y_n^{\pm 1}]$. There is an
epimorphism of (super)algebras $\phi : K[G]\to A$, defined by
$c_{ij}\mapsto\delta_{ij}x_i$ if $i\leq m$ and $c_{ij}\mapsto\delta_{ij}y_i$ otherwise. 
The algebra $A$ can be obviously
identified with $K[T]$; that is $A$ has an obvious $T$-supermodule
structure. Moreover, for any $\lambda\in X(T)$ the epimorphism
$\phi$ maps $K[G]_{\lambda}$ to $A_{\lambda}$.
\begin{ter}(Chevalley's restriction theorem)
The restriction of $\phi$ on $R$ is a monomorphism.
\end{ter}
\begin{proof} Let $f\in R\setminus 0$. It suffices to prove that the
image of a leading term of $f$ is non-zero. This is clear if the
image of the leading summand of the companion of $f$ is non-zero.
Since $A$ is an integral domain, the images of leading summands of $C_1,\ldots , C_m,
D_1,\ldots , D_{n-1}$ and $Ber(C)$ are all non-zeroes. Use the above
product representation of $f*$ to finish the proof.
\end{proof}
\bigskip

Denote the images of $C_r$ and $D_r$ in $A$ by $c_r$ and $d_r$
correspondingly. It is clear that
$$c_r(x_1,\ldots , x_m, y_1,\ldots , y_n)=c_r(x|y)=\sum_{0\leq
i\leq \min\{r, m\}}(-1)^{r-i}\sigma_i(x)p_{r-i}(y),$$ where
$\sigma_i(x)$ and $p_j(y)$ are elementary and complete symmetric
polynomials, respectively.

A subalgebra of $A$ consisiting of polynomials
$f(x|y)=f(x_1,\ldots , x_m, y_1,\ldots , y_n)$, symmetric in variables
$x$ and $y$ separately, and such that
$\frac{d}{dt}(f|_{x_1=y_1=t})=0$, is denoted by $A_s$. In the case of $char
K=0$, this algebra has already been considered in \cite{ktr}
and \cite{stem} (it was called an algebra of pseudosymmetric and
supersymmetric polynomials, respectively). It was proved in \cite{stem} that $A_s$
is generated by the elements $c_r$. Since $\phi(R_{pol})\subseteq
A_s$ (see \cite{ktr}), it implies that $\phi(R_{pol})=A_s$ and
$R_{pol}$ is generated by  the elements $C_r$. 
In the case of positive characteristic, the generators of $R_{pol}$
are unknown. However, in any characteristic it holds
$\phi(R_{pol})\subseteq A_s$ by the same arguments as in
\cite{ktr}.

As in \cite{z}, denote the largest polynomial subsupermodule of
$H^0_-(\lambda)$ by $\nabla(\lambda)$. Recall that
$\nabla(\lambda)\neq 0$ if and only if $L(\lambda)\neq 0$ is polynomial if
and only if $\lambda\in X(T)^{++}$. The subset $X(T)^{++}\subseteq X(T)^+$ was
completely described in \cite{bkl1}. Let ${\cal V}=\{V\}$ be a
collection of polynomial $G$-supermodules. The collection ${\cal V}$ is called {\it good}
if and only if for any $\lambda\in X(T)^{++}$ there is $V\in {\cal V}$ such
that $\lambda$ is the highest weight of $V$. For example, a
collection of all simple polynomial $G$-supermodules is good. The
collection $\{\nabla(\lambda)\}_{\lambda\in X(T)^{++}}$ is also
good.

Denote by $X(T)_{\geq 0}$ the subset ${\bf N}^{m+n}\subseteq
X(T)={\bf Z}^{m+n}$ and by $X(T)_{\geq 0}^{+}$ the intersection
$X(T)_{\geq 0}\bigcap X(T)^{+}$.
\begin{lm}
If $\mu\in X(T)^+_{\geq 0}$, then the set $(\mu]\bigcap
X(T)^+_{\geq 0}$ is finite.
\end{lm}
\begin{proof} It follows from 
$$
|(\mu]\bigcap
X(T)^+_{\geq 0}|\leq (\mu_1+1)^m\times (|\mu_+|+\mu_{m+1}+1)^n.
$$
\end{proof}

\begin{ter}
The algebra $R_{pol}$ is generated (as a vector space) by $Tr(V)$,
where $V$ runs over any good collection of $G$-supermodules.
\end{ter}
\begin{proof} Consider $f\in R_{pol}$. As above, let $M$ be a $G\times
G$-subsupermodule generated by $f$ and $M\subseteq
O_{\Gamma}(K[G])$ for a finitely generated ideal $\Gamma$.
In that case $M\subseteq V_t$ for some $t\leq k$ and $M\bigcap V_{k+1}=0$, 
where $\{V_i\}_{i\geq 0}$ is a Donkin-Koppinen filtration of
$O_{\Gamma}(K[G])$. Denote by $\Gamma'$ an ideal generated by (finitely many) $\mu$
such that $f_{\mu}\neq 0$. It is obvious that
$\Gamma'$ is also generated by leading weights of $f$. Since
$M|_{1\times G}$ contains a polynomial $G$-subsupermodule
(generated by $f$), the right hand side
factor $H_-^0(\lambda_t)$ in a quotient $V_t/V_{t+1}\simeq
V_-(\lambda_t)^*\otimes H_-^0(\lambda_t)$ has non-zero polynomial part. In other
words, 
there is $V\in {\cal V}$ such that its highest weight is $\lambda_t$
and $\lambda_t\in X(T)^{++}$. Again, $Tr(V)\in V_t\setminus
V_{t+1}$ and for a non-zero scalar $a$, all weights of polynomial
invariant $f-aTr(V)$ belong to $(\Gamma'\bigcap X(T)^+_{\geq
0})\setminus\{\lambda_t\}$. Lemma 5.7 concludes the proof.
\end{proof}
\bigskip

Let $V$ be a $G$-supermodule with a basis as in Lemma 5.3. The
(Laurent) polynomial $\chi(V)=\phi(Tr(V))=\sum_{1\leq i\leq
t}(-1)^{|v_i|}\phi(c_{ii})$ is called a {\it formal supercharacter}
of $V$. It is clear that $\chi(V)=\sum_{\lambda\in X(T)}(\dim
(V_{\lambda})_0-\dim (V_{\lambda})_1)x^{\lambda_+}y^{\lambda_-}$,
where $x^{\lambda_+}=\prod_{1\leq i\leq m}x_i^{\lambda_i}$ and
$y^{\lambda_-}=\prod_{m+1\leq i\leq m+n}y_{i-m}^{\lambda_i}$.

\begin{ex}(see \cite{zm1})
Set $m=n=1$. Every simple polynomial $GL(1|1)$-supermodule is at most
two-dimensional. More precisely, set $X(T)^{++}=\bigcup_{r\geq 0}
X(T)^{++}_r$, where
$X(T)^{++}_r=\{\lambda=(\lambda_1|\lambda_2)\in
X(T)^{++}||\lambda|=r\}$. If $p|r$, then $X(T)^{++}_r=\{(i,
r-i)|0\leq i\leq r\}$ and $L(i)=L(i|r-i)$ is (even) one-dimensional supermodule. 
In particular, $Tr(L(i))=x_1^i y_1^{r-i}\in A_s$. Otherwise, $X(T)^{++}_r=\{(i, r-i)|1\leq i\leq
r\}$ and $L(i)=L(i|r-i)$ is a two-dimensional supermodule. By
definition, its highest vector is even and the second basic vector
is odd of weight $(i-1|r-i+1)$. Thus $Tr(L(i))=x_1^i
y_1^{r-i}-x_1^{i-1} y_1^{r-i+1}\in A_s$.
\end{ex}
A homogeneous polynomial $f(x|y)=\sum_{\lambda\in
X(T)}a_{\lambda}x^{\lambda_+} y^{\lambda_-}$ is said to be
$p$-{\it balanced} if and only if for every $\lambda$ such that $a_{\lambda}\neq 0$
and every $1\leq i\leq m < j\leq m+n$ it satisfies
$p|(\lambda_i+\lambda_j)$. In the example above, all $Tr(L(i))$ are
$p$-balanced, provided $p|r$. In general, if
$f$ is $p$-balanced and symmetric in variables $x$ and $y$ separately, then
$f\in A_s$. Moreover, $p$-balanced polynomials from $A_s$ form a
subalgebra $A_s(p)$.
\begin{pr}
The algebra $A_s(p)$ is generated by the elements $\sigma_i(x)^p,
\sigma_j(y)^p$ for $1\leq i\leq m$ and $1\leq j\leq n$, and by the elements
$u_k(x|y)=\sigma_m(x)^k\sigma_n(y)^{p-k}$ for $0 <k <p$. Moreover,
$A_s(p)\subseteq\phi(R_{pol})$.
\end{pr}
\begin{proof} The elements $c_{ij}^p$ , where $|i|+|j|=0\pmod 2$,
generate a Hopf subsuperalgebra of $K[G]$. In fact,
$\delta_G(c_{ij}^p)=\sum_{k, |k|=|i|}c_{ik}^p\otimes c_{kj}^p$ and 
$s_G(c_{ij}^p)=s_G(c_{ij})^p$. 
If $\sigma_i(C_{00})$ and 
$\sigma_j(C_{11})$ are $i$-th and $j$-th coefficients of the
characteristic polynomials of the blocks $C_{00}$ and $C_{11}$
respectively, then 
$\sigma_i(C_{00})^p$ and $\sigma_j(C_{11})^p$ belong to $R_{pol}$ for $1\leq i\leq m$ and $1\leq j\leq n$.
Moreover, $\phi(\sigma_i(C_{00})^p)=\sigma_i(x)^p$ and
$\phi(\sigma_j(C_{11})^p)=\sigma_j(y)^p$. Furthermore, the element
$\sigma_n(C_{11})^p$ is group-like. In particular, it can be
naturally identified with a one-dimensional simple $G$-supermodule
$L(0|p^n)$ (see Proposition 6.2 of \cite{z}). Therefore,
$L(k^m|(p-k)^n)=L(0|p^n)\otimes Ber(E)^{k}$ is also
one-dimensional simple $G$-supermodule, whose formal character
coincides with $u_k(x|y)$. By Theorem 6.5 of \cite{bkl1},
$L(k^m|(p-k)^n)$ is a polynomial $G$-supermodule. It follows that
$\sigma_n(C_{11})^{p}Ber(C)^k=Tr(L(k^m|(p-k)^n))\in R_{pol}$ and
$u_k\in\phi(R_{pol})$.

Finally, consider a homogeneous polynomial $f(x|y)=\sum_{\mu\in
X(T)_{\geq 0}, |\mu|=r}f_{\mu}x^{\mu_+}y^{\mu_-}$ from $A_s(p)$.
Without a loss of generality we can assume that there is
$\lambda\in X(T)^+_{\geq 0}$ such that if $f_{\mu}\neq 0$, then
$\mu$ is equal to $\lambda$ up to a permutation of its first $m$
and last $n$ coordinates separately. Since $f(x|y)$ is
$p$-balanced, it follows that $\lambda=p\lambda' +(k^m| (-k)^n)$,
where $\lambda'\in X(T)^+_{\geq 0}$ such that $\lambda'_{m+n} > 0$ and
$0 < k <p$. In other words, $f(x|y)=u_k(x|y)g(x|y)^p$ and $g(x|y)$
is symmetrical in variables $x$ and $y$ separately.
\end{proof}

\begin{hyp}
The algebra $R_{pol}$ is generated by the elements 
$$C_r,
\sigma_i(C_{00})^p, \sigma_j(C_{11})^p, \sigma_n(C_{11})^p
Ber(C)^k$$ for $r\geq 0, 1\leq i\leq m, 1\leq j\leq n$ and $0 < k < p$.
\end{hyp}
Utilizing Proposition 5.1, the above conjecture would follow from the following one. 

\begin{hyp}
The algebra $A_s$ is generated over $A_s(p)$ by the elements $c_r$ for 
$r\geq 0$.
\end{hyp}

\begin{center}
\bf Acknowledgements
\end{center}
This work was supported by RFFI 10-01-00383a and by INDAM.


\begin{thebibliography}{99}
\bibitem{ber} F.A.Berezin, {\em Introduction to syperanalysis},
D., Reidel Publishing Co., Dordrecht, 1987.Expanded translation
from the Russian: {\em Introduction to algebra and analysis with
anticommuting variables.} Moscow State University, Moscow, 1983.
V.P.Palamodov, ed.
\bibitem{bkl} J.Brundan and A.Kleshchev, {\em
Modular representations of the supergroup $Q(n)$, I},  J.Algebra,
260(2003), N 1, 64-98.
\bibitem{bkl1} J.Brundan and J.Kujawa, {\em A new proof of the
Mullineux conjecture}, J.Algebraic.Combin., 18(2003), 13-39.
\bibitem{bd} I.Bucur and A.Deleanu, {\em Introduction to the theory of categories and functors},
A Wiley-Interscience Publ., 1968.
\bibitem{cps} E.Cline, B.Parshall and L.Scott, {\em Finite dimensional algebras and highest weight categories},
J.reine angew.Math., 391(1988), 85-99.
\bibitem{dr} V.Dlab and C.R.Ringel, {\em The module theoretical approach to quasi-hereditary
algebras}, London Math.Soc.Lecture Note Ser.,168, 1992.
\bibitem{gm} S.I.Gelfand and Yu.I.Manin, {\em Homological algebra}, Algebra, V, Encyclopaedia Math. Sci. 38,
1994.
\bibitem{jan} J.Jantzen, {\em Representations of algebraic groups},
Academic Press, Inc., 1987.
\bibitem{ktr} I.Kantor and I.Trishin, {\em The algebra of polynomial invariants of the adjoint representation of
the Lie superalgebra $gl(m|n)$}, Comm.Algebra, 25(7): 2039-2070,
1997.
\bibitem{kv} H.M. Khudaverdian and Th.Th.Voronov, {\em
Berezinians, exterior powers and recurrent sequences},
Lett.Math.Phys., 74(2005), N 2, 201-228.
\bibitem{kop} M.Koppinen, {\em Good bimodule filtrations for coordinate
rings}, J.London.Math.Soc(2), 30(1984), N 2, 244-250.
\bibitem{zm1} F.Marko and A.N.Zubkov, {\em Schur
superalgebras in characteristic $p$}, Algebras and Representation
Theory, 9(2006), N 1, 1-12.
\bibitem{zm} F.Marko and A.N.Zubkov, {\em Schur
superalgebras in characteristic $p$, II}, Bulletin of London
Math.Soc., 38(2006), 99-112.
\bibitem{mu} N.J.Muir, {\em Polynomial representations of the general linear
Lie superalgebras}, Ph.D. Thesis, University of London, 1991.
\bibitem{stem} J.R.Stembridge, {\em A characterization of supersymmetric
polynomials}, J.Algebra, 95(1985), N 2, 439-444.
\bibitem{z} A.N.Zubkov, {\em Some properties of general linear supergroups and of
Schur superalgebras}, (Russian), Algebra Logika, 45(2006), N 3,
257-299.
\bibitem{z1} A.N.Zubkov, {\em Affine quotients of supergroups},
Transformation Groups, 14(2009), N 3, 713-745.
\end{thebibliography}
\end{document}